\documentclass{amsart}

\usepackage{amssymb}
\usepackage{graphicx}

\newtheorem{theorem}{Theorem}[section]
\newtheorem{lemma}[theorem]{Lemma}
\newtheorem{conjecture}[theorem]{Conjecture}

\theoremstyle{definition}

\theoremstyle{remark}
\newtheorem{remark}[theorem]{Remark}

\newtheorem{corollary}[theorem]{Corollary}

\numberwithin{equation}{section}

\begin{document}

\title[Restricted averaging operators]{Restricted averaging operators to cones over finite fields}

\author{Doowon Koh, Chun-Yen Shen, and Seongjun Yeom}
\address{Department of Mathematics\\
Chungbuk National University \\
Cheongju Chungbuk 28644, Korea}
\email{koh131@chungbuk.ac.kr}

\address{Department of Mathematics\\
National Central University \\
Chungli,  32054 Taiwan}
\email{chunyshen@gmail.com}

\address{Department of Mathematics\\
Chungbuk National University \\
Cheongju Chungbuk 28644, Korea}
\email{mathsj@chungbuk.ac.kr}

\thanks{The first author was supported by  the research grant of   Basic Science Research Program through the National Research Foundation of Korea  funded by the Ministry of Education, Science and Technology (NRF-2015R1A1A1A05001374) and the second author was supported by the NSC, through grant NSC102-2115-M-008-015-MY2.}


\thanks{}





\begin{abstract} We investigate  the sharp $L^p\to L^r$ estimates for the restricted averaging operator $A_C$ over the cone $C$ of the $d$-dimensional vector space $\mathbb F_q^d$ over the finite field $\mathbb F_q$ with $q$ elements. The restricted averaging operator $A_C$ for the cone $C$ is defined by the relation that $A_Cf=f\ast \sigma |_C$, where $\sigma$ denotes the normalized surface measure on the cone $C$, and $f$ is a complex valued function on the space $\mathbb F_q^d$ with the normalized counting measure $dx$. In the previous work \cite{KY15}, the sharp boundedness of $A_C$ was obtained in odd dimensions $d\ge 3$ but   partial results were only given in even dimensions $d\ge 4.$ In this paper we prove the optimal estimates in even dimensions $d\ge 6$ in the case when the cone $C\subset \mathbb F_q^d$ contains a $d/2$ dimensional subspace.

\end{abstract}

\maketitle

\section{Introduction}

Let $T$ be an operator on the class of Schwarz functions $f: \mathbb R^d \to \mathbb C.$
The main question in harmonic analysis is to determine the exponents $1\le p, r\le \infty$ such that the following inequality holds:
\begin{equation}\label{eq1} \|Tf\|_{L^r(\mathbb R^d)} \le C \|f\|_{L^p(\mathbb R^d)}\end{equation}
where the constant $C>0$ is independent of the Schwarz functions $f.$
For example, when  $Tf$ is the Fourier transform of $f$, the Hausdorff-Young inequality states that the inequality \eqref{eq1} holds for  $1\le p\le 2$ and $ 1/p+1/r=1.$\\

 Another interesting question is to decide whether $Tf$ can be meaningfully restricted to a surface $V\subset \mathbb R^d$ or not.
More precisely, we are interested in finding exponents $1\le p,r\le \infty$ such that the following restriction estimate holds:
$$\|Tf\|_{L^r(V, d\nu)} \le C \|f\|_{L^p(\mathbb R^d)}$$
where $d\nu$ denotes a surface measure on $V \subset \mathbb R^d.$
Clearly, the answer to this question relies on the surface $V$ and the operator $T.$
To indicate that $Tf$ is a function restricted to the surface $V,$ we write $T_Vf$ for $Tf.$
This problem is referred to as the restriction problem for the surface $V.$
When  $T_Vf=\widehat{f},$ it is well known as the Fourier restriction problem for $V$ which was initially introduced by E.M. Stein in 1967.
In particular, many researchers have  made  much effort on solving the conjecture on the Fourier restriction problem for the sphere, the paraboloid, and the cone.
The complete answers are known for the parabola and the circle in two dimensions, and for the cones in three and four dimensions(see \cite{Zy74, Ba85, Wo01}).
However, the conjecture is still open in higher dimensions and some new ideas are needed to completely understand the Fourier restriction phenomena.
For the background and  recent progress on the Fourier restriction problem, we  refer readers to \cite{Ta04, Wo03, Bo91, Fe70,  To75,   B91, Ta03,  Gu15, Gu16}.\\

It has been believed that new approaches are needed to obtain further improvement on the Fourier restriction problem.
It will be helpful to see the matter in a different point of view.
Based on this mind, Mockenhaupt and Tao \cite{MT04} initially studied the Fourier restriction problem in the finite field setting.
Their work has been improved by other researchers (see \cite{IK10, Le13, Le14, LL10}). Other interesting problems in harmonic analysis have been formulated and studied in the finite field setting (for example, see \cite{Ca06, CSW, EOT}).\\

It is also important to grasp the fundamental phenomena which appear in restricting  operators to an appropriate surface.
One may study the restriction problem related to certain operators which are different from the Fourier transformation.
The authors in \cite{GGIPS13} provided some size information about  convolution functions restricted to any affine subspace in $\mathbb R^d.$
In the finite field setting, the author in \cite{Ko15} initially investigated and obtained the sharp $L^p\to L^r$ mapping properties for the restricted averaging operator  to any algebraic curve in two dimensional vector spaces over finite fields. This result was deduced by applying  the sharp Fourier restriction theorem on curves in two dimensions.
This work was extended to higher dimensional algebraic varieties such as  the paraboloid, the sphere, and the cone.
Indeed, using more delicate Fourier decay estimate, the authors in \cite{KY15}  established the optimal $L^p\to L^r$  estimates of the restricted averaging operator over regular varieties such as the paraboloid and the sphere in all dimensions, and the  cone in odd dimensions. In addition, they obtained certain weak-type estimates for the cone in even dimensions. In this paper,  we shall establish the sharp strong-type estimates for the cone in even dimensions $d\ge 6$ in the specific case when the cone $C$ contains a $d/2$-dimensional subspace.\\

\subsection{Review of the discrete Fourier analysis}
 After reviewing the definition of  the restricted averaging problem for the cone in the  finite field setting,  our main result will be clearly stated.
Before we proceed with this, let us introduce some notation and basic concepts of the discrete Fourier analysis.
Let $\mathbb F_q^d$ be the $d$-dimensional vector space over a finite filed $\mathbb F_q$ with $q$ elements.
We shall always assume that $q$ is a power of odd prime.
We endow the space $\mathbb F_q^d$ with the normalized counting measure $dx.$
We write $(\mathbb F_q^d, dx)$ for the vector space $\mathbb F_q^d$ with the normalized counting measure.
On the other hand, the dual space of $(\mathbb F_q^d, dx)$ will be denoted by $(\mathbb F_q^d, dm)$ which is equipped with the counting measure $dm.$
Thus, if  $ f: (\mathbb F_q^d, dx) \to \mathbb C$, and $g:(\mathbb F_q^d, dm) \to \mathbb C,$ then the notation of norms is  used as follows: for $1\le p<\infty,$
$$ \|f\|_{L^p(\mathbb F_q^d, dx)}= \left( \frac{1}{q^d} \sum_{x\in \mathbb F_q^d} |f(x)|^p\right)^{1/p} ~~\mbox{and} \quad \|g\|_{L^p(\mathbb F_q^d, dm)}= \left( \sum_{m\in \mathbb F_q^d} |f(m)|^p\right)^{1/p}. $$
Also recall that $ \|f\|_{L^\infty(\mathbb F_q^d, dx)} =\max\limits_{x\in \mathbb F_q^d} |f(x)|$ and $ \|g\|_{L^\infty(\mathbb F_q^d, dm)} =\max\limits_{m\in \mathbb F_q^d} |g(m)|.$
The cone $C\subset (\mathbb F_q^d, dx), d\ge 3,$ is defined by the set
\begin{equation}\label{defcone} C=\{x\in \mathbb F_q^d: x_1^2+x_2^2+ \cdots+ x_{d-2}^2 =x_{d-1} x_d\}.\end{equation}
Mockenhaupt and Tao \cite{MT04} gave the complete answer to the restriction problem for the cone $C$ in three dimensions.
We endow the cone $C$ with the normalized surface measure $\sigma$ which is defined by the relation
$$ \int_C f(x) ~d\sigma(x) :=\frac{1}{|C|} \sum_{x\in C} f(x),$$
where $|C|$ denotes the cardinality of the cone $C\subset \mathbb F_q^d,$ and $f:(\mathbb F_q^d, dx)\to \mathbb C.$ In other words, the mass of each point of the cone $C$ is considered as $1/|C|.$
\begin{remark}\label{remark1}Since $d\sigma(x) =\frac{q^d}{|C|} C(x)~dx,$  the normalized surface measure $\sigma$ on the cone $C$ can be identified with a function $ \frac{q^d}{|C|}  C(x)$ on $(\mathbb F_q^d, dx)$,
where we write $C(x)$ for the characteristic function $\chi_C$ on the cone $C.$
Namely, we shall identify a set $E$ with the characteristic function $\chi_E, $ which allows us to use simple notation. \end{remark}

Let  $g$ be a complex-valued function on $(\mathbb F_q^d, dm).$ Then  $\widehat{g}$, the Fourier transform of $g$ is defined on the dual space $(\mathbb F_q^d, dx)$ as follows:
$$ \widehat{g}(x)= \int_{\mathbb F_q^d}  \chi(-m\cdot x) \,g(m)~ dm=\sum_{m\in \mathbb F_q^d} \chi(-m\cdot x)\, g(m),$$
where $\chi$ denotes a  nontrivial additive character of $\mathbb F_q$ and  $m\cdot x$ is the usual dot-product notation.
Given a function $f:(\mathbb F_q^d, dx) \to \mathbb C$,  the inverse Fourier transform of $f$, denoted by $f^\vee$, is defined on $(\mathbb F_q^d, dm)$ and it takes the following form
$$ f^\vee(m) =\int_{\mathbb F_q^d} \chi(m\cdot x) \,f(x)~dx = \frac{1}{q^d} \sum_{x\in \mathbb F_q^d} \chi(m\cdot x) \,f(x).$$
We stress that the choice of $\chi$ does not change our results in this paper as long as $\chi$ is a nontrivial additive character of $\mathbb F_q.$
Recall that the orthogonality relation of $\chi$ states that
$$\sum_{x\in \mathbb F_q^d} \chi( m\cdot x) = \left\{ \begin{array}{ll} 0  \quad&\mbox{if}~~ m\neq (0,\dots,0)\\
q^d  \quad &\mbox{if} ~~m= (0,\dots,0), \end{array}\right.$$
Observe that the Plancherel theorem states
$|| f^\vee\|_{L^2(\mathbb F_q^d, dm)} = \|f\|_{L^2(\mathbb F_q^d, dx)},$ which yields
$$ \sum_{m\in \mathbb F_q^d} |f^\vee(m)|^2 = \frac{1}{q^d} \sum_{x\in \mathbb F_q^d} |f(x)|^2.$$

\subsection{ Restricted averaging problem and statement of main result}
Given two functions $f, h: (\mathbb F_q^d, dx)\to \mathbb C,$  the convolution function $f\ast h$ is defined on $(\mathbb F_q^d, dx)$ as follows:
$$ f\ast h(y)=\int_{\mathbb F_q^d} f(y-x)\, h(x)~dx=\frac{1}{q^d} \sum_{x\in \mathbb F_q^d} f(y-x)\, h(x)\quad \mbox{for}~~y\in (\mathbb F_q^d, dx).$$
It clearly follows that $(f\ast h)^\vee(m)=f^\vee(m) \,h^\vee(m)$ for $m\in (\mathbb F_q^d, dm).$  Replacing the function $h$ by the normalized surface measure $\mu$ on an algebraic variety $V\subset (\mathbb F_q^d, dx),$
the averaging operator $A$ can be defined by
$$ Af(y)=f\ast \mu(y)=\int_V f(y-x) ~d\mu(x) := \frac{1}{|V|} \sum_{x\in V} f(y-x),$$
where both $f$ and $Af$ are defined on $(\mathbb F_q^d, dx).$ In the finite field setting,   Carbery-Stones-Wright \cite{CSW} initially studied the $L^p\to L^r$ estimates for the averaging operator. Much attention has been given to this problem (for example, see \cite{Ko, KS11, KSS}).\\

As a variant of the averaging operator $A$ over $(V, \mu) \subset (\mathbb F_q^d, dx)$,  a restricted averaging operator $A_V$ to $V$ is defined  by restricting $Af=f\ast \mu$ to the variety $V$. Namely, we have
$ A_V f=Af|_V=f\ast d\mu|_V.$
 Then the restricted averaging problem is to determine $1\le p,r \le \infty$ such that the following restricted averaging inequality holds:
\begin{equation}\label{A(pr)} \|A_Vf\|_{L^r(V, \mu)} \le C \|f\|_{L^p(\mathbb F_q^d, dx)},\end{equation}
where the constant $C>0$ is independent of the functions $f$ and the size of the underlying finite field $\mathbb F_q.$
This problem was proposed in \cite{Ko15} where the sharp restricted averaging inequality was established in the case when the variety $V$ is any curve on plane.
Such a sharp result was  obtained by a direct application of  the complete solution to the Fourier restriction problem for curves in two dimensions. The authors in \cite{KY15} observed from the Fourier decay estimate  that the optimal restricted averaging inequalities can be obtained if the variety $V\subset (\mathbb F_q^d, dx)$ satisfies the following two conditions:
\begin{equation}\label{defregular}|V|\sim q^{d-1} \quad \mbox{and}\quad  \max_{m\ne (0,\ldots,0)} | V^{\vee}(m)| \lesssim q^{-(d+1)/2}.\end{equation}
Here, and throughout this paper, we write $ E(x)$ for the characteristic function $\chi_{E}$ on the set $E \subset \mathbb F_q^d.$
Also recall that $X \lesssim Y$ is used to denote that there exists $C>0$ independent of $q=|\mathbb F_q|$ such that $X\le C Y.$ In addition, $X\sim Y$ means $X\lesssim Y \lesssim X.$
We shall write $A_V(p\to r)\lesssim 1$ if the restricted averaging inequality \eqref{A(pr)} holds.\\

A variety $V\subset (\mathbb F_q^d, dx)$ satisfying the conditions \eqref{defregular} is called a regular variety. Typical regular varieties are the paraboloids and the spheres with nonzero radius.
When the restricted operator $A_V$ is related to a non-regular variety $V\subset (\mathbb F_q^d, dx)$,  it may  not be a simple problem to find the sharp restricted averaging inequality, because the optimal results can not be obtained by simply applying the Fourier decay estimate. Therefore, it would be interesting to prove sharp restricted averaging inequalities for non-regular varieties. The cone $C\subset (\mathbb F_q^d, dx)$ defined as in \eqref{defcone} has unusual structures in that it is not a regular variety in even dimensions $d\ge 4$, but it is a regular variety in odd dimensions $d\ge 3$ (see Corollary 4.3 in \cite{KY15}).  For this reason, we are interested in establishing the sharp restricted averaging problem on cones $C\subset (\mathbb F_q^d, dx)$ in even dimensions $d\ge 4.$ The necessary conditions for the boundedness of $A_C(p\to r)$ were given in \cite{KY15}. For example, the lemma below follows immediately from Lemma 2.1 in \cite{KY15}.
\begin{lemma}\label{LemN} Let $\sigma$ denote the normalized surface measure on the cone $C\subset (\mathbb F_q^d, dx).$ Suppose that the restricted averaging estimate
$$ \|A_Cf\|_{L^{r}(C, \sigma)} \lesssim \|f\|_{L^{p}(\mathbb F_q^d, dx)}$$
holds for all function $f$ on $(\mathbb F_q^d, dx).$ Then the following two statements are true:

\begin{enumerate}
\item If   the cone $C$ does not contain any subspace $H$ with $|H|> q^{\frac{d-1}{2}}$, then  $\left(\frac{1}{p},\, \frac{1}{r}\right)$ must lie on the convex hull of points $(0,0), (0,1), \left(\frac{d-1}{d}, \, 1\right)$ and $\,P_0:=\left(\frac{d-1}{d}, \, \frac{1}{d}\right).$
\item  If  the cone $C$ contains a $d/2$-dimensional subspace $H,$ then $\left(\frac{1}{p},\, \frac{1}{r}\right)$ lies on the convex hull of points $(0,0), (0,1),$
$\left(\frac{d-1}{d}, \, 1\right), \,
P_1:=\left(\frac{d-1}{d}, \, \frac{1}{d-2}\right)$ and $\, P_2:=\left(\frac{d^2-3d+2}{d^2-2d+2}, \, \frac{d-2}{d^2-2d+2}\right).$
\end{enumerate}
\end{lemma}

From the nesting property of norms and the interpolation with the trivial $L^\infty\to L^\infty$ estimate, to prove that the necessary conditions are in fact sufficient, it suffices to obtain the critical point $P_0.$ In addition, to prove the optimal results in the case when the cone $C$ contains a $d/2$-dimensional subspace,  it will be enough to obtain the critical points $P_1$ and $P_2.$ In fact, when $d\ge 3$ is odd, the critical point $P_0$ was obtained in \cite{KY15}, which gives the complete answer to the restricted averaging problem for cones in odd dimensions. On the other hand, when $d\ge 4$ is even, it is in general impossible to obtain the point $P_0,$ because the cone $C$ may contain a $d/2$-dimensional subspace. As we shall see, the cone $C\subset (\mathbb F_q^d, dx)$ contains a $d/2$-dimensional subspace if $d=4k+2$ for $k\in \mathbb N$, or if $-1\in \mathbb F_q$ is a square number and $d\ge 4$ is even. In this case, to settle the restricted averaging problem for cones, it suffices to obtain the critical points $P_1$ and $P_2.$  In this paper, we shall establish the critical points except for dimension four. As a consequence, we give  complete answers to the restricted averaging problems for cones in even dimensions $d\ge 6$ in the case when the cone $C$ contains a $d/2$-dimensional subspace. More precisely, our main theorem is as follows:

\begin{theorem}\label{main} Let $A_C$ be the restricted averaging operator associated with the cone $C\subset (\mathbb F_q^d, dx)$ defined as in \eqref{defcone}.
Suppose that $\sigma$ denotes the normalized surface measure on the cone $C.$ Then, if $d\ge 6$ is even, we have
\begin{equation}\label{eq1.3}
\|A_Cf\|_{L^{d-2}(C, \sigma)} \lesssim \|f\|_{L^{ \frac{d}{d-1}}(\mathbb F_q^d, dx)}
\end{equation}
and if $d\ge 4$ is even, then we have
\begin{equation}\label{eq1.4}
\|A_Cf\|_{L^{\frac{d^2-2d+2}{d-2} }(C, \sigma)} \lesssim \|f\|_{L^{\frac{d^2-2d+2}{d^2-3d+2} }(\mathbb F_q^d, dx)}
\end{equation}
\end{theorem}

In \cite{KY15},  it was proved that the inequality \eqref{eq1.3} holds if $d\ge 4$ is even and the test functions $f$ are characteristic functions on $(\mathbb F_q^d, dx).$
It was also proved in \cite{KY15} that the dual estimate of the inequality \eqref{eq1.4} holds for all characteristic test functions on the cone $C$ in even dimensions $d\ge 4.$ Hence,  Theorem \ref{main} provides the improved endpoint estimates in even dimensions $d\ge 6.$ The estimate \eqref{eq1.4} gives a partial improvement in four dimensions.

\subsection{Remark on sharpness of Theorem \ref{main}}
As mentioned before, we see from Theorem \ref{main} and Lemma \ref{LemN}  that if $d\ge 6$ is even and the cone $C\subset (\mathbb F_q^d, dx)$ contains a $d/2$-dimensional subspace, then
$A_C(p\to r)\lesssim 1$ if and only if $\left(\frac{1}{p},\, \frac{1}{r}\right)$ is contained in the convex hull of the points
$(0,0), (0,1),$
$\left(\frac{d-1}{d}, \, 1\right), \,
P_1:=\left(\frac{d-1}{d}, \, \frac{1}{d-2}\right)$ and $\, P_2:=\left(\frac{d^2-3d+2}{d^2-2d+2}, \, \frac{d-2}{d^2-2d+2}\right).$
Let $\eta$ denote the quadratic character of $\mathbb F_q^*.$
In addition, assume that $H$ denotes a maximal subspace contained in the cone $C\subset (\mathbb F_q^d, dx).$
It is well known that for even dimensions $d\ge 4$ we have
\begin{equation}\label{Hestimate} |H|=\left\{\begin{array}{ll} q^{\frac{d}{2}} \quad&\mbox{if}~~  \eta(-1)=\left(\eta(-1)\right)^{\frac{d}{2}}\\
  q^{\frac{d-2}{2}} \quad &\mbox{if}~~ \eta(-1)=-\left(\eta(-1)\right)^{\frac{d}{2}}\end{array} \right.\end{equation}
(for example, see Lemma 2.1 in \cite{Vi12}). Thus, if $d=4k+2$ for $k\in \mathbb N$, or $ -1\in \mathbb F_q$ is a square number and $d\ge 4$ is even, then the cone $C$ contains a subspace $H$ with $|H|=q^{\frac{d}{2}}.$ In conclusion, Theorem \ref{main} provides the complete mapping properties of the restricted averaging operator $A_C$ in the case when $d=4k+2$ for $k\in \mathbb N$, or $ -1\in \mathbb F_q$ is a square number and $d\ge 6$ is even.

\begin{remark} Notice from Theorem \ref{main} that to settle the restricted averaging problem for the cone $C$  in the case when $d=4$ and $-1\in \mathbb F_q$ is a square number, we only need to prove  the inequality \eqref{eq1.3} for $d=4.$ However, it looks a hard problem  and we leave this as an open question.  \end{remark}

From \eqref{Hestimate} we see that if $-1\in \mathbb F_q$ is not a square number and $d=4k$ for $k\in \mathbb N,$ then $q^{\frac{d-2}{2}}$ is the cardinality of a maximal subspace lying in the  cone $C\subset (\mathbb F_q^d, dx).$ Combining this fact with the first conclusion of Lemma \ref{LemN}, we may conjecture the following.

\begin{conjecture}
Let $C\subset (\mathbb F_q^d, dx)$ be the cone. Assume that $d=4k$ for $k\in\mathbb N$, and $-1\in \mathbb F_q$ is not a square number. Then we have $A_C(p\to r)\lesssim 1$ if and only if
$\left(\frac{1}{p},\, \frac{1}{r}\right)$  lies on the convex hull of points $(0,0), (0,1), \left(\frac{d-1}{d}, \, 1\right)$ and $\,P_0:=\left(\frac{d-1}{d}, \, \frac{1}{d}\right).$
\end{conjecture}

As seen before, in order to establish this conjecture, it will be enough to obtain the critical point $P_0.$

\subsection{Contents of the remain parts of this paper}
 The remain parts of this paper will be organized as follows.
In Section \ref{sec2},  we introduce preliminary key lemmas which play a crucial role in proving Theorem \ref{main}.
The proof of the inequalities \eqref{eq1.3} and \eqref{eq1.4} in Theorem \ref{main} will be given in Sections \ref{sec3} and \ref{sec4}, respectively.

\section{Preliminaly lemmas} \label{sec2}
In this section, we collect several lemmas most of which are implicitly contained in \cite{KY15}.
Let us denote by $A_C^*$ the adjoint operator of the restricted averaging operator $A_C$ to the cone $C\subset (\mathbb F_q^d, dx).$
Since $A_Cf =f\ast \sigma |_C$, it follows  that
 $$< A_C f,~ h>_{L^2(C, \sigma)} =<f, ~A_C^*h>_{L^2(\mathbb F_q^d, dx)},$$
where we recall that $\sigma$ is the normalized surface measure on the cone $C.$
From this, we see that the adjoint operator $A_C^*$ is given by
$$A_C^* h(y)= \frac{q^d}{|C|^2} \sum_{x\in C} C(x-y) h(x)$$
where $h:(C, \sigma)\to \mathbb C$ and $y\in (\mathbb F_q^d, dx).$
Since $C= -C,$ we can alternatively write that for all functions
$h: (\mathbb F_q^d, dx) \to \mathbb C$ with $h(x)=0$ for $x\in \mathbb F_q^d \setminus C.$
\begin{equation}\label{dualoper} A_C^*h=\frac{q^{2d}}{|C|^2}~ (hC)\ast C =\frac{q^{2d}}{|C|^2}~ h\ast C.\end{equation}

We aim to find the exponents $1\le p,r\le \infty$ such that
$$ \|A_C f\|_{L^r(C, \sigma)}  \lesssim \|f\|_{L^p(\mathbb F_q^d, dx).}$$
By duality, this equality is same as the following inequality
$$ \|A_C^* h\|_{ L^{p^\prime}(\mathbb F_q^d, dx)} \lesssim \|h\|_{L^{r^\prime}(C, \sigma)},$$
where $p^\prime=p/(p-1)$ and $r^\prime=r/(r-1).$

\subsection{ Decomposition of the restricted averaging operator}
Define a function $K$ on $(\mathbb F_q^d, dm)$ by
\begin{equation}\label{defK} K(m)=\sigma^\vee (m) -\delta_{\bf 0} (m),\end{equation}
where $\delta_{\bf 0}(m)=1$ for $m=(0,\ldots, 0)$ and $0$ otherwise.
Then we can write $ \sigma(x)=\widehat{K}(x) + \widehat{\delta_{\bf 0}} (x) =\widehat{K}(x) +1.$
Thus, the restricted averaging operator $A_C$ to the cone $C$ can be decomposed by
\begin{equation}\label{decK} A_C f= f\ast \sigma= f\ast 1 + f\ast \widehat{K}.\end{equation}
Observe from the definition of $K$ that $K(m)=\sigma^\vee(m)$ for $m\ne (0, \ldots, 0)$ and $ K(0, \ldots, 0)=0.$
Then the following lemma is a direct result from Corollary 4.4 in \cite{KY15}.
\begin{lemma}\label{Lem2.1} Let $\sigma$ be the normalized surface measure on the cone $C\subset (\mathbb F_q^d, dx)$ defined as in \eqref{defcone}. Define  $K(m)=\sigma^\vee(m)-\delta_0(m)$ for $m\in (\mathbb F_q^d, dm)$ and $\Gamma(\xi)=\xi_1^2+\xi_2^2+\cdots+\xi_{d-2}^2-4\xi_{d-1}\xi_d$ for $\xi=(\xi_1, \ldots, \xi_d)\in \mathbb F_q^d.$
If the dimension $d\ge 4$ is even and $m\ne (0, \ldots,0) \in \mathbb F_q^d,$ then we have
$$ |K(m)|=|\sigma^\vee(m)|\sim \left\{\begin{array}{ll} q^{-\frac{(d-2)}{2}} ~~ &\mbox{for} ~~\Gamma(m)=0 \\
                                                                                    q^{- \frac{d}{2}}  ~~ &\mbox{for}~~ \Gamma(m)\ne 0. \end{array} \right.$$
                                                                                    In addition, we have $K(0, \ldots, 0) =0.$
\end{lemma}

The following Fourier restriction estimate was given in Lemma 3.1 in \cite{KY15}.
\begin{lemma}\label{Lem2.2} Let $\sigma$ be the normalized surface measure on the cone $C \subset (\mathbb F_q^d, dx).$
Then we have
$$ \|\widehat{g}\|_{L^2(C, \sigma)} \lesssim q^{\frac{1}{2}} \|g\|_{L^2(\mathbb F_q^d, dm)}\quad \mbox{for all} ~~g:(\mathbb F_q^d, dm) \to \mathbb C.$$
\end{lemma}
\begin{proof} By duality, it is enough to prove the following extension estimate:
$$ \|(f\sigma)^{\vee} \|_{L^2(\mathbb F_q^d, dm)} \sim q^{\frac{1}{2}} \|f\|_{L^2(C, \sigma)} \quad \mbox{for all} ~~f: C \to \mathbb C.$$
Since $\sigma(x)=\frac{q^d}{|C|} C(x) $ and $|C|\sim q^{d-1},$  the Plancherel theorem yields
\begin{align*}\|(f\sigma)^{\vee} \|_{L^2(\mathbb F_q^d, dm)} &= \frac{q^d}{|C|}\| (fC)^\vee \|_{L^2(\mathbb F_q^d, dm)} = \frac{q^d}{|C|} \|fC\|_{L^2(\mathbb F_q^d, dx)} \\
&=\frac{q^{d/2}}{|C|^{1/2}} \|f\|_{L^2(C, \sigma)} \sim q^{\frac{1}{2}} \|f\|_{L^2(C, \sigma)}.\end{align*}
\end{proof}

The following lemma was also given in Lemma 4.5 in \cite{KY15}.
\begin{lemma}\label{Lem2.3}Let $C^*=\{m\in \mathbb F_q^d:  \Gamma(m)=0\}.$ If the dimension, $d\ge 4$ is even, then we have
$$ \sum_{m\in C^*} |E^\vee(m)|^2 \lesssim q^{-d-1}|E|+ q^{-\frac{3d}{2}} |E|^2\quad\mbox{for all}~~E\subset (\mathbb F_q^d, dx).$$
\end{lemma}

We shall invoke the following result.
\begin{lemma}\label{Biglem}
Let $\sigma$ be the normalized surface measure on the cone
$C \subset (\mathbb F_q^d, dx).$ Then if $d\ge 4$ is even, the estimate
$$ \sum_{m\in \mathbb F_q^d \setminus \{(0,\ldots,0)\}}
\left|E^{\vee}(m) ~\sigma^{\vee}(m)\right|^2 \lesssim \min\left\{ q^{-2d+2} |E|,~ q^{-2d+1}|E|+ q^{\frac{-5d+4}{2}}|E|^2\right\}$$
holds for all sets $E\subset (\mathbb F_q^d, dx).$
\end{lemma}
\begin{proof}
Let $\Gamma$ be the function defined as in the statement of Lemma \ref{Lem2.1}.
We write
$$ \sum_{m \ne (0,\ldots,0)}
\left|E^{\vee}(m) ~\sigma^{\vee}(m)\right|^2
=\sum_{\substack{m\ne (0,\ldots,0)\\
: \Gamma(m)\ne 0}}
\left|E^{\vee}(m) ~\sigma^{\vee}(m)\right|^2 + \sum_{\substack{m\ne (0,\ldots,0)\\
: \Gamma(m)= 0}}
\left|E^{\vee}(m) ~\sigma^{\vee}(m)\right|^2.$$

Applying Lemma \ref{Lem2.1}, Lemma \ref{Lem2.3}, and  the Plancherel theorem, we see that
\begin{align*}
\sum_{m \ne (0,\ldots,0)}
\left|E^{\vee}(m) ~\sigma^{\vee}(m)\right|^2
&\lesssim q^{-d} \sum_{m\in \mathbb F_q^d} |E^{\vee}(m)|^2 + q^{-d+2} \sum_{\Gamma(m)=0} |E^{\vee}(m)|^2\\
&\lesssim q^{-d} q^{-d}|E| + \left( q^{-2d+1} |E| + q^{\frac{-5d+4}{2}}
|E|^2\right)\\
&\lesssim q^{-2d+1}|E| + q^{\frac{-5d+4}{2}} |E|^2.
\end{align*}
On the other hand, we also see from Lemma \ref{Lem2.1} and the Plancherel theorem that
$$ \sum_{m \ne (0,\ldots,0)}
\left|E^{\vee}(m) ~\sigma^{\vee}(m)\right|^2 \lesssim q^{-d+2} \sum_{m\in \mathbb F_q^d} |E^\vee(m)|^2 =q^{-2d+2}|E|.$$
Putting all estimates together, we obtain the statement of the lemma.
\end{proof}

The following lemma will play an important role in deriving our main result.
\begin{lemma}\label{Lem2.5} Let $K$ be defined as in \eqref{defK}. If the dimension $d\ge 4$ is even,  then the estimates
\begin{equation}\label{eq2.4} \|E\ast \widehat{K}\|_{L^\infty(C, \sigma)} \lesssim \frac{|E|}{q^{d-1}} \end{equation}
and
\begin{equation} \label{eq2.5}
\|E\ast \widehat{K}\|_{L^2(C, \sigma)} \lesssim
\min\left\{ q^{\frac{-2d+3}{2}} |E|^{\frac{1}{2}}, ~ q^{-d+1} |E|^{\frac{1}{2}} + q^{\frac{-5d+6}{4}} |E|\right\}
\end{equation}
hold for all $E \subset (\mathbb F_q^d, dx).$
\end{lemma}
\begin{proof} To prove the inequality \eqref{eq2.4},  observe from Remark \ref{remark1} that
$$ \max_{y\in \mathbb F_q^d} |\widehat{K}(y)| =\max_{y\in \mathbb F_q^d} | \sigma(y)-1| = \max_{y\in \mathbb F_q^d} \left|\frac{q^d C(y)}{|C|} -1\right| \le \frac{q^d}{|C|} \sim q,$$
where we used that $|C|\sim q^{d-1}.$
Then it follows that for any $x\in C,$
$$  |E\ast \widehat{K}(x)| \le \left( \max_{y\in \mathbb F_q^d} |\widehat{K}(y)| \right) \frac{1}{q^d} \sum_{y\in \mathbb F_q^d} |E(x-y)| \sim \frac{|E|}{q^{d-1}},$$
and we obtain the inequality \eqref{eq2.4}. Next, in order to prove the inequality \eqref{eq2.5} holds, it will be enough to show that
\begin{equation}\label{eq2.6} \|E\ast \widehat{K}\|_{L^2(C, \sigma)}^2 \lesssim \min\left\{ q^{-2d+3} |E|, ~ q^{-2d+2} |E| + q^{\frac{-5d+6}{2}}
|E|^2 \right\} .\end{equation}
Since $ E\ast \widehat{K} = \widehat{E^\vee K},$ we see from Lemma \ref{Lem2.2} that
$$ \|E\ast \widehat{K}\|_{L^2(C, \sigma)}^2 =\|\widehat{ E^\vee K}\|_{L^2(C, \sigma)}^2 \le q  \| E^\vee K\|_{L^2(\mathbb F_q^d, dm)}^2.$$
By the definition of $K$, the right-hand side is written by
$$ q \sum_{m\ne (0,\ldots,0)} \left|E^\vee(m)~ \sigma^\vee(m)\right|^2.$$
Applying Lemma \ref{Biglem} to this estimate, we obtain the inequality \eqref{eq2.6}. Thus, we complete the proof of the inequality \eqref{eq2.5}.
\end{proof}

The following result is much weaker than \eqref{eq2.5} of Theorem \ref{Lem2.5}, but it is useful to apply in practice.
\begin{corollary} \label{cor2.6} If $d\ge 4$ is even, then we have
$$ \|E\ast \widehat{K}\|_{L^2(C, \sigma)}
\lesssim q^{-d+1} |E|^{\frac{d+2}{2d}}$$
for all $E\subset (\mathbb F_q^d, dx).$
\end{corollary}
\begin{proof}
Notice that the estimate \eqref{eq2.5} of Lemma \ref{Lem2.5} implies that
 if $d\ge 4$ is even, the estimate
\begin{equation}\label{eq2.7}
\|E\ast \widehat{K}\|_{L^2(C, \sigma)} \lesssim
\left\{ \begin{array}{ll} q^{\frac{-2d+3}{2}} |E|^{\frac{1}{2}}\quad\mbox{if}~~ q^{\frac{d}{2}} \le |E| \le q^d\\
   q^{\frac{-5d+6}{4}} |E| \quad\mbox{if}~~ q^{\frac{d-2}{2}} \le |E| \le q^{\frac{d}{2}}\\
   q^{-d+1} |E|^{\frac{1}{2}} \quad\mbox{if}~~  1\le |E| \le q^{\frac{d-2}{2}} \end{array}\right.
\end{equation}
holds for all $E\subset (\mathbb F_q^d, dx),$ which in turn implies the conclusion of the corollary.
\end{proof}

  The following result will be used to deduce the estimate \eqref{eq1.3} of Theorem \ref{main}.

\begin{lemma} \label{Lem2.7} If the dimension $d\ge 6$ is even, the estimate
$$\|E\ast \widehat{K}\|_{L^{\frac{d-2}{2}}(C, \sigma)} \lesssim q^{-d+1}|E|^{\frac{d-2}{d}}$$
holds for all $E\subset (\mathbb F_q^d, dx).$
\end{lemma}
\begin{proof}
Since $2\le \frac{d-2}{2} <\infty$ for $d\ge 6$, the statement follows immediately by  interpolating the estimate \eqref{eq2.4} of Lemma \ref{Lem2.5} and the conclusion of Corollary \ref{cor2.6}.
\end{proof}

\subsection{Decomposition of the dual restricted averaging operator}
We shall decompose the dual operator $A_C^*$ defined as in \eqref{dualoper}.
We define a function $M$ on $(\mathbb F_q^d, dm)$ by
\begin{equation}\label{defM} M(m)=C^\vee(m)-\frac{|C|}{q^d} \delta_0(m)\quad\mbox{for}~~m\in (\mathbb F_q^d, dm).\end{equation}
Then for each $x\in (\mathbb F_q^d, dx)$ we can write
\begin{equation}\label{fC} C(x)=\widehat{M}(x) +\frac{|C|}{q^d} \widehat{\delta_0}(x)= \widehat{M}(x)+ \frac{|C|}{q^d}.\end{equation}
Namely, the characteristic function on the cone $C$ is same as the function $\widehat{M} + \frac{|C|}{q^d}.$
Recall from \eqref{dualoper} that we can write
$$ A_C^*h= \frac{q^{2d}}{|C|^2} h\ast C$$
where $h$ is a function supported on $C.$ Thus, $A_C^*$ can be decomposed as
$$ A_C^*h= \frac{q^{2d}}{|C|^2} h\ast \widehat{M} + \frac{q^{d}}{|C|} h\ast 1.$$

The following lemma plays a crucial role in proving the inequality \eqref{eq1.4} of Theorem \ref{main}.
\begin{lemma}\label{estimateM}
Let $M$ be the function defined as in \eqref{defM}. If the dimension, $d\ge 4,$ is even, then the estimates
\begin{equation}\label{M1}
\|F\ast \widehat{M}\|_{L^\infty(\mathbb F_q^d, dx)} \lesssim \frac{|F|}{q^d}
\end{equation}
and
\begin{equation}\label{M2}
\|F\ast \widehat{M}\|_{L^2(\mathbb F_q^d, dx)} \lesssim
\min\left\{q^{-d}|E|^{\frac{1}{2}},~ q^{\frac{-2d-1}{2}} |E|^{\frac{1}{2}}+q^{\frac{-5d}{4}} |E|\right\}
\end{equation}
hold for all $ F \subset (C, \sigma).$
\end{lemma}

\begin{proof}
To prove the inequality \eqref{M1}, we notice from Young's inequality for convolutions that
$$ \|F\ast \widehat{M}\|_{L^\infty(\mathbb F_q^d, dx)} \le \|F\|_{L^1(\mathbb F_q^d, dx)}~\| \widehat{M}\|_{L^\infty(\mathbb F_q^d, dx)}= \frac{|F|}{q^d}\,\| \widehat{M}\|_{L^\infty(\mathbb F_q^d, dx)}.$$
Since $|C|\sim q^{d-1}$, it is clear from \eqref{fC} that $\| \widehat{M}\|_{L^\infty(\mathbb F_q^d, dx)} \lesssim 1.$ Thus, the inequality \eqref{M1} holds. Next, we shall prove the inequality \eqref{M2}. Squaring the both sides of the inequality \eqref{M2}, it suffices to show that
\begin{equation} \label{eq2.12}
\|F\ast \widehat{M}\|^2_{L^2(\mathbb F_q^d, dx)} \lesssim
\min\left\{q^{-2d}|E|,~ q^{-2d-1} |E|+q^{\frac{-5d}{2}} |E|^2\right\}\quad\mbox{for all}~~F\subset C.
\end{equation}
By the Plancherel theorem, it follows that
$$ \|F\ast \widehat{M}\|^2_{L^2(\mathbb F_q^d, dx)}
=\|F^\vee M\|^2_{L^2(\mathbb F_q^d, dm)} =\sum_{m\in \mathbb F_q^d} |F^\vee(m)~ M(m) |^2.$$
By the definition of $M$  in \eqref{defM}, it is clear that  $M(m)=C^\vee(m)$ for $m\ne (0,\ldots,0)$ and $M(0,\ldots,0)=0.$ Also recall that the normalized surface measure $\sigma$ on the cone $C$ can be identified with a function $\sigma(x)=\frac{q^d}{|C|} C(x) \sim q C(x).$ It follows that
$$\|F\ast \widehat{M}\|^2_{L^2(\mathbb F_q^d, dx)}
=\sum_{m\ne (0,\ldots,0)} |F^\vee(m)~ C^\vee(m) |^2 \sim q^{-2}\sum_{m\ne (0,\ldots,0)} |F^\vee(m) \sigma^\vee(m)|^2.$$
Then the estimate \eqref{eq2.12} is obtained by using Lemma \ref{Biglem}.
Thus, the proof is complete.
\end{proof}

By a direct computation, the following result is obtained from \eqref{M2} of Lemma \ref{estimateM}.
\begin{corollary}\label{Cor2.9} Let $d\ge 4$ be even. Then the estimate
\begin{equation}\label{eq2.13}
\|F\ast \widehat{M}\|_{L^2(\mathbb F_q^d, dx)} \lesssim
\left\{ \begin{array}{ll} q^{-d} |F|^{\frac{1}{2}}\quad\mbox{if}~~ q^{\frac{d}{2}} \le |F| \lesssim q^{d-1}\\
   q^{\frac{-5d}{4}} |F| \quad\mbox{if}~~ q^{\frac{d-2}{2}} \le |F| \le q^{\frac{d}{2}}\\
   q^{\frac{-2d-1}{2}} |F|^{\frac{1}{2}} \quad\mbox{if}~~  1\le |F| \le q^{\frac{d-2}{2}} \end{array}\right.
\end{equation}
holds for all $F\subset (C, \sigma).$
\end{corollary}

We shall need the following estimates.

\begin{lemma} \label{Lem2.10} If $d\ge 6$ is even, then the following estimate holds for all
$F\subset (C, \sigma)$:
\begin{equation}\label{eq2.14}
\|F\ast \widehat{M}\|_{L^{\frac{d^2-2d+2}{2d}}(\mathbb F_q^d, dx)} \lesssim
\frac{|F|^{\frac{d^2-4d+6}{d^2-2d+2}}} {q^{\frac{d^3-2d^2+4d}{d^2-2d+2}}}.
\end{equation}

On the other hand, in the dimension four, the estimate
\begin{equation}\label{eq2.15}\|F\ast \widehat{M}\|_{L^{\frac{10}{3}}(\mathbb F_q^4, dx)} \lesssim
\left\{ \begin{array}{ll} q^{-4} |F|^{\frac{7}{10}}\quad\mbox{if}~~ q^{2} \le |F| \lesssim q^{3}\\
   q^{-\frac{23}{5}} |F| \quad\mbox{if}~~ q \le |F| \le q^2\\
   q^{-\frac{43}{10}} |F|^{\frac{7}{10}} \quad\mbox{if}~~  1\le |F| \le q \end{array}\right.
\end{equation}
holds for all $F\subset (C, \sigma).$
\end{lemma}

\begin{proof}
First, let us prove the estimate \eqref{eq2.14}.
By a direct comparison, we see that the estimate \eqref{eq2.13} of Corollary \ref{Cor2.9} implies that if $d\ge 4$ is even, then
\begin{equation}\label{eq2.16}
\|F\ast \widehat{M}\|_{L^2(\mathbb F_q^d, dx)}
\lesssim q^{\frac{-2d-1}{2}} |F|^{\frac{d+2}{2d}} \quad \mbox{for all}~~ F\subset (C, \sigma).
\end{equation}
Recall from \eqref{M1} of Lemma \ref{estimateM} that if $d\ge 4$ is even, then
\begin{equation}\label{eq2.17}
\|F\ast \widehat{M}\|_{L^\infty(\mathbb F_q^d, dx)} \lesssim q^{-d}|F|
\quad \mbox{for all}~~ F\subset (C, \sigma).
\end{equation}
Since $2< \frac{d^2-2d+2}{2d} <\infty$ for $d\ge 6$, the estimate \eqref{eq2.14} follows by interpolating \eqref{eq2.16} and \eqref{eq2.17}.\\

Next, to prove the estimate \eqref{eq2.15}, notice that if $d=4,$ then \eqref{M1} of Lemma \ref{estimateM} and \eqref{eq2.13} of Corollary \ref{Cor2.9} yield the following two estimates:
for every $F\subset (C,\sigma)$
$$ \|F\ast \widehat{M}\|_{L^\infty(\mathbb F_q^4, dx)} \lesssim q^{-4}|F| $$
and
\begin{equation}\label{L2fine} \|F\ast \widehat{M}\|_{L^2(\mathbb F_q^4, dx)} \lesssim
\left\{ \begin{array}{ll} q^{-4} |F|^{\frac{1}{2}}\quad\mbox{if}~~ q^{2} \le |F| \lesssim q^{3}\\
   q^{-5} |F| \quad\mbox{if}~~ q \le |F| \le q^{2}\\
   q^{-\frac{9}{2}} |F|^{\frac{1}{2}} \quad\mbox{if}~~  1\le |F| \le q. \end{array}\right. \end{equation}
 Since $2< \frac{10}{3} <\infty,$  the estimate  \eqref{eq2.15} of Lemma \ref{Lem2.10} follows by interpolating the above two estimates.
\end{proof}

\section{The proof of the inequality \eqref{eq1.3} in Theorem \ref{main}}\label{sec3}
In this section, we restate and prove the first part of Theorem \ref{main}.
\begin{theorem}\label{mainagain1} Let $A_C$ be the restricted averaging operator associated with the cone $C\subset (\mathbb F_q^d, dx)$ defined as in \eqref{defcone}.
Suppose that $\sigma$ denotes the normalized surface measure on the cone $C.$ Then, if $d\ge 6$ is even,  the estimate
$$
\|A_Cf\|_{L^{d-2}(C, \sigma)} \lesssim \|f\|_{L^{ \frac{d}{d-1}}(\mathbb F_q^d, dx)}
$$
holds for all functions $f$ on $\mathbb F_q^d.$
\end{theorem}
\begin{proof}
We aim to show that the estimate
$$ \|f\ast \sigma\|_{L^{d-2}(C, \sigma)} \lesssim \|f\|_{L^{ \frac{d}{d-1}}(\mathbb F_q^d, dx)}=q^{-d+1} \left(\sum_{x\in \mathbb F_q^d} |f(x)|^{\frac{d}{d-1}} \right)^{\frac{d-1}{d}}$$
holds for all functions $f$ on $\mathbb F_q^d.$
Without loss of generality, we may assume that
$f$ is a non-negative real-valued function and
\begin{equation}\label{easy1} \sum_{x\in \mathbb F_q^d}
f(x)^{\frac{d}{d-1}} =1.\end{equation}
Then $\|f\|_\infty \leq 1$ and so we may assume that $f$ is written by a step function
\begin{equation}\label{easy2} f(x)=\sum_{i=0}^\infty 2^{-i} E_i(x),\end{equation}
where $E_i$'s are pairwise disjoint subsets of $\mathbb F_q^d.$ Combining  \eqref{easy1} with \eqref{easy2}, we also assume that
\begin{equation} \label{easy3} \sum_{j=0}^\infty 2^{-\frac{dj}{d-1}}|E_j|=1\quad\mbox{and so} ~~ |E_j|\leq 2^{\frac{dj}{d-1}}~~\mbox{for all}~~j=0,1,\cdots.\end{equation}
Thus, to complete the proof we only need to show that the estimate
$$ \|f\ast \sigma\|_{L^{d-2}(C, \sigma)} \lesssim q^{-d+1} $$
holds for all functions $f$ on $\mathbb F_q^d$ satisfying the assumptions \eqref{easy1},\eqref{easy2}, \eqref{easy3}.
As seen in \eqref{decK}, we can write $f\ast \sigma=f\ast 1 + f\ast \widehat{K},$
and thus our problem is reduced to showing that the following two estimates hold:
\begin{equation}\label{Y1} \|f\ast 1\|_{L^{d-2}(C, \sigma)} \lesssim q^{-d+1}\end{equation}
and
\begin{equation}\label{Y2} \|f\ast \widehat{K}\|_{L^{d-2}(C, \sigma)} \lesssim q^{-d+1},\end{equation}
where the function $K$ on $(\mathbb F_q^d, dm)$ is defined as in \eqref{defK}.  Since $\max\limits_{x\in C} |f\ast 1(x)| \le \|f\|_{L^1(\mathbb F_q^d, dx)},$ the estimate \eqref{Y1} can be obtained by observing
$$ \|f\ast 1\|_{L^{d-2}(C, \sigma)} \le \|f\|_{L^1(\mathbb F_q^d, dx)} \le
\|f\|_{L^{\frac{d}{d-1}}(\mathbb F_q^d, dx)} =q^{-d+1},$$
where we used the assumption \eqref{easy1}.
It remains to prove the estimate \eqref{Y2} which is in turn written by
\begin{equation}\label{YR1}
\mbox{A}:=q^{2d-2}~ \|(f\ast \widehat{K}) (f\ast \widehat{K}) \|_{L^{\frac{d-2}{2}}(C, \sigma)} \lesssim 1.
\end{equation}
Using \eqref{easy2}, we see that
\begin{align*}
\mbox{A} &\le q^{2d-2} \sum_{k=0}^\infty \sum_{j=0}^\infty 2^{-k-j}
\|(E_k\ast \widehat{K}) (E_j\ast \widehat{K}) \|_{L^{\frac{d-2}{2}}(C, \sigma)} \\
&\sim q^{2d-2} \sum_{k=0}^\infty \sum_{j=k}^\infty 2^{-k-j}
\|(E_k\ast \widehat{K}) (E_j\ast \widehat{K}) \|_{L^{\frac{d-2}{2}}(C, \sigma)},
\end{align*}
where the last line is obtained by the symmetry of $k,j.$
Using \eqref{eq2.4} of Lemma \ref{Lem2.5} and Lemma \ref{Lem2.7}, we see that
\begin{align*}\mbox{A}&\lesssim ~ q^{2d-2} \sum_{k=0}^\infty \sum_{j=k}^\infty 2^{-k-j}
\|E_k\ast \widehat{K}\|_{L^\infty(C, \sigma)}~\|E_j\ast \widehat{K} \|_{L^{\frac{d-2}{2}}(C, \sigma)}\\
&\lesssim ~\sum_{k=0}^\infty \sum_{j=k}^\infty 2^{-k-j} |E_k|~|E_j|^{\frac{d-2}{d}}. \end{align*}
By \eqref{easy3}, we conclude that
\begin{align*}\mbox{A}&\lesssim \sum_{k=0}^\infty\sum_{j=k}^\infty 2^{-k-j}|E_k| ~2^{\frac{(d-2)j}{d-1}}=
\sum_{k=0}^\infty 2^{-k} |E_k| \left(\sum_{j=k}^\infty 2^{\frac{-j}{d-1}}\right)\\
& \sim \sum_{k=0}^\infty 2^{-k} |E_k| ~2^{\frac{-k}{d-1}} =\sum_{k=0}^\infty {2^{-\frac{dk}{d-1}}}|E_k|=1.\end{align*}
Thus, we complete the proof.
\end{proof}

\section{The proof of the inequality \eqref{eq1.4} in Theorem \ref{main}}\label{sec4}
We shall provide the complete proof of the second part of Theorem \ref{main} which can be restated as follows.
\begin{theorem}\label{mainagain2} Let $A_C$ be the restricted averaging operator associated with the cone $C\subset (\mathbb F_q^d, dx)$ defined as in \eqref{defcone}.
Suppose that $\sigma$ denotes the normalized surface measure on the cone $C.$ Then, if $d\ge 4$ is even, the estimate
\begin{equation}\label{E2again}
\|A_Cf\|_{L^{\frac{d^2-2d+2}{d-2} }(C, \sigma)} \lesssim \|f\|_{L^{\frac{d^2-2d+2}{d^2-3d+2} }(\mathbb F_q^d, dx)}
\end{equation}
holds for all function $f$ on $\mathbb F_q^d.$ \end{theorem}

\begin{proof} By duality, it suffices to prove that if $d\ge 4$ is even, then
$$ \|A^*_Ch\|_{L^{\frac{d^2-2d+2}{d}}(\mathbb F_q^d, dx)}
\lesssim \|h\|_{L^{\frac{d^2-2d+2}{d^2-3d+4}}(C, \sigma)}\quad\mbox{for all}~~h:(C,\sigma)\to \mathbb C,$$
where we recall from \eqref{dualoper} that for $x\in (\mathbb F_q^d, dx)$
$$ A^*_Ch(x) =\frac{q^{2d}}{|C|^2} (h\ast C)(x).$$
Put $r=\frac{d^2-2d+2}{d}$ and $p=\frac{d^2-2d+2}{d^2-3d+4}.$ Then our task is to show that the estimate
\begin{equation}\label{eq4.2}
\left\|\frac{q^{2d}}{|C|^2} (h\ast C)\right\|_{L^{r}(\mathbb F_q^d, dx)}
\lesssim \|h\|_{L^{p}(C, \sigma)}
\end{equation}
holds for all $h:(C,\sigma)\to \mathbb C.$
As usual, we may assume that $h$ is a nonnegative real valued function supported on the cone $C.$ By normalization of $h$, we also assume that
\begin{equation}\label{eq4.3} \sum_{x\in C} |h(x)|^p =1. \end{equation}
Furthermore, we may assume that the function $h$ can be written by a step function
\begin{equation}\label{eq4.4}
h(x)= \sum_{j=0}^\infty 2^{-j} F_j(x) \quad \mbox{for}~~x\in C,
\end{equation}
where $F_j$'s are pairwise disjoint subsets of $C.$ From \eqref{eq4.3} and \eqref{eq4.4}, we also assume
\begin{equation}\label{eq4.5} \sum_{k=0}^\infty 2^{-pk} |F_k| =1.
\end{equation}
Hence, it is natural to assume that for every $k=0,1,\ldots,$
\begin{equation}\label{eq4.6}
 |F_k| \le 2^{pk}.
\end{equation}
With the above assumptions on $h$,  our problem is reduced to showing that
if $d\ge 4$ is even, then
$$\left\|\frac{q^{2d}}{|C|^2} (h\ast C)\right\|_{L^{r}(\mathbb F_q^d, dx)}
\lesssim \|h\|_{L^{p}(C, \sigma)}.$$
Now recall from \eqref{defM} and \eqref{fC} that the characteristic function on the cone $C$ is written by
$$C(x)= \widehat{M}(x)+ \frac{|C|}{q^d}\quad \mbox{for}~~x\in (\mathbb F_q^d, dx),$$
where the function $M$ on $(\mathbb F_q^d, dm)$ is defined by
 $M(m)=C^\vee(m)-\frac{|C|}{q^d} \delta_0(m).$
 Then, to complete the proof, it will be enough to show that if $d\ge 4$ is even, then we have
 \begin{equation}\label{FTerm}
 \left\|\frac{q^{d}}{|C|} (h\ast 1)\right\|_{L^{r}(\mathbb F_q^d, dx)}
\lesssim \|h\|_{L^{p}(C, \sigma)}
 \end{equation}
 and
 \begin{equation}\label{STerm}
 \left\|\frac{q^{2d}}{|C|^2} (h\ast \widehat{M})\right\|_{L^{r}(\mathbb F_q^d, dx)}
\lesssim \|h\|_{L^{p}(C, \sigma)}:=|C|^{-\frac{1}{p}} \left(\sum_{x\in C} |h(x)|^p\right)^{\frac{1}{p}},
 \end{equation}
 where $r=\frac{d^2-2d+2}{d},~p=\frac{d^2-2d+2}{d^2-3d+4},$ and the function $h$ satisfies \eqref{eq4.3},\eqref{eq4.4},\eqref{eq4.5},\eqref{eq4.6}.
The estimate \eqref{FTerm} simply follows by using Young's inequality for convolution functions. Indeed, it follows that
\begin{align*}\left\|\frac{q^{d}}{|C|} (h\ast 1)\right\|_{L^{r}(\mathbb F_q^d, dx)}
&\le \frac{q^{d}}{|C|} \|h\|_{L^1(\mathbb F_q^d, dx)} ~\|1\|_{L^{r}(\mathbb F_q^d, dx)} \\
&=\frac{q^d}{|C|} \|h\|_{L^1(\mathbb F_q^d, dx)}
= \|h\|_{L^1(C, \sigma)} \le \|h\|_{L^p(C, \sigma)},\end{align*}
where the last inequality follows because $dx$ is the normalized counting measure and $1<p.$ Thus,  it remains to prove the estimate \eqref{STerm}.
Using \eqref{eq4.3} with the facts that $|S|\sim q^{d-1}$ for $d\ge 4$ and $p= \frac{d^2-2d+2}{d^2-3d+4},$ the estimate \eqref{STerm} can be rewritten by
$$ q^{\frac{d^3-2d^2+3d}{d^2-2d+2}} \|h\ast \widehat{M}\|_{L^r(\mathbb F_q^d, dx)} \lesssim 1.$$
Since  $\|h\ast \widehat{M}\|^2_{L^r(\mathbb F_q^d, dx)}
=\|(h\ast \widehat{M})(h\ast \widehat{M})\|_{L^{\frac{r}{2}}(\mathbb F_q^d, dx)},$ the above estimate is equivalent to the estimate
$$  q^{\frac{2d^3-4d^2+6d}{d^2-2d+2}} \|(h\ast \widehat{M})(h\ast \widehat{M})\|_{L^{\frac{r}{2}}(\mathbb F_q^d, dx)} \lesssim 1,$$ which we must prove. By \eqref{eq4.4} and Minkowski's inequality, the left hand side of the above inequality is dominated by
\begin{align*} & q^{\frac{2d^3-4d^2+6d}{d^2-2d+2}} \sum_{k=0}^\infty \sum_{j=0}^\infty 2^{-k-j} \|(F_k\ast \widehat{M})(F_j\ast \widehat{M})\|_{L^{\frac{r}{2}}(\mathbb F_q^d, dx)}\\
\sim \,& q^{\frac{2d^3-4d^2+6d}{d^2-2d+2}} \sum_{k=0}^\infty \sum_{j=k}^\infty 2^{-k-j} \|(F_k\ast \widehat{M})(F_j\ast \widehat{M})\|_{L^{\frac{r}{2}}(\mathbb F_q^d, dx)}, \end{align*}
where the last line is obtained by the symmetry of $k, j.$
Thus, our final task to complete the proof  is to show that if $d\ge 4$ is even, then we have
\begin{equation}\label{Finalproof}
\mbox{B}:=q^{\frac{2d^3-4d^2+6d}{d^2-2d+2}} \sum_{k=0}^\infty \sum_{j=k}^\infty 2^{-k-j} \|(F_k\ast \widehat{M})(F_j\ast \widehat{M})\|_{L^{\frac{r}{2}}(\mathbb F_q^d, dx)} \lesssim 1,
\end{equation}
where $r=\frac{d^2-2d+2}{d}$ and we assume that \eqref{eq4.5} and \eqref{eq4.6} hold with $p=\frac{d^2-2d+2}{d^2-3d+4}.$
In the following subsections, we shall prove the estimate \eqref{Finalproof}
in the case when  $d\ge 6$ and  $d=4,$ respectively, and so  the proof of Theorem \ref{mainagain2} will be complete.
\subsection{Proof of the estimate \eqref{Finalproof} for even dimensions $d\ge 6$} Assume that $d\ge 6$ is even.
From \eqref{M1} of Lemma \ref{estimateM} and \eqref{eq2.14} of Lemma \ref{Lem2.10}, we see that
\begin{align}\label{eq4.10}\nonumber
\mbox{B} &\le q^{\frac{2d^3-4d^2+6d}{d^2-2d+2}} \sum_{k=0}^\infty \sum_{j=k}^\infty 2^{-k-j} \|F_k\ast \widehat{M}\|_{L^\infty(\mathbb F_q^d, dx)} ~
\|F_j\ast \widehat{M}\|_{L^{\frac{r}{2}}(\mathbb F_q^d, dx)} \\
&\lesssim  \sum_{k=0}^\infty \sum_{j=k}^\infty 2^{-k-j} |F_k|\, |F_j|^{\frac{d^2-4d+6}{d^2-2d+2}}
\end{align}
Since $ p=\frac{d^2-2d+2}{d^2-3d+4},$  it follows from \eqref{eq4.6} and \eqref{eq4.5} that
$$|F_j|\le 2^{\frac{(d^2-2d+2)j}{d^2-3d+4}} ~~ \mbox{for all}~~ j=0,1,\ldots \quad ~~{and}~~ \sum_{k=0}^\infty 2^{-\frac{(d^2-2d+2)k}{d^2-3d+4}} |F_k|
=1.$$
Using these facts, we conclude that
\begin{align*}
\mbox{B} &\lesssim \sum_{k=0}^\infty \sum_{j=k}^\infty 2^{-k-j} |F_k|\, 2^{\frac{(d^2-4d+6)j}{d^2-3d+4}}
=\sum_{k=0}^\infty 2^{-k} |F_k| \left(  \sum_{j=k}^\infty  2^{\frac{(-d+2)j}{d^2-3d+4}}\right)\\
&\sim \sum_{k=0}^\infty 2^{-k} |F_k| \,2^{\frac{(-d+2)k}{d^2-3d+4}}
=\sum_{k=0}^\infty 2^{-\frac{(d^2-2d+2)k}{d^2-3d+4}} |F_k|
=1.
\end{align*}
 Thus the estimate \eqref{Finalproof} holds for even dimensions $d\ge 6.$
\begin{remark} Recall that to deduce the inequality \eqref{eq4.10} we used the estimate \eqref{eq2.14} of Lemma \ref{Lem2.10} which was proved only for even dimension $d\ge 6.$ However, if $d=4$, we can not apply the estimate \eqref{eq2.14} of Lemma \ref{Lem2.10} and so we need to take a different approach to prove the estimate \eqref{Finalproof} for $d=4.$
\end{remark}

\subsection{Proof of the estimate \eqref{Finalproof} for $d=4$}
We aim to show that
\begin{equation}\label{eq4.11}\mbox{B}:=q^{\frac{44}{5}} \sum_{k=0}^\infty \sum_{j=k}^\infty 2^{-k-j} \|(F_k\ast \widehat{M})(F_j\ast \widehat{M})\|_{L^{\frac{5}{4}}(\mathbb F_q^4, dx)} \lesssim 1,\end{equation}
where the following conditions hold:
\begin{equation}\label{eq4.12} \sum_{k=0}^\infty 2^{-\frac{5k}{4}} |F_k|=1 \end{equation}
and
\begin{equation}\label{eq4.13} |F_k|\le 2^{\frac{5k}{4}}\quad \mbox{for all}~~j=0,1,\ldots.
\end{equation}
By  H\"older's inequality, we have
$$\mbox{B}\le q^{\frac{44}{5}} \sum_{k=0}^\infty \sum_{j=k}^\infty 2^{-k-j} \|(F_k\ast \widehat{M})\|_{L^{\frac{10}{3}}(\mathbb F_q^4, dx)}  \|(F_j\ast \widehat{M})\|_{L^2(\mathbb F_q^4, dx)}.$$
Since  $F_k \subset C\subset \mathbb F_q^4$ for $k=0,1,\ldots,$ and $|C|\sim q^3,$ we have
\begin{align*}
\mbox{B}\le \,& q^{\frac{44}{5}} \sum_{\substack{k=0\\
:1\le |F_k|\le q}}^\infty \sum_{j=k}^\infty 2^{-k-j} \|(F_k\ast \widehat{M})\|_{L^{\frac{10}{3}}(\mathbb F_q^4, dx)}  \|(F_j\ast \widehat{M})\|_{L^2(\mathbb F_q^4, dx)}\\
&+ q^{\frac{44}{5}} \sum_{\substack{k=0\\
: q\le |F_k|\le q^2}}^\infty \sum_{j=k}^\infty 2^{-k-j} \|(F_k\ast \widehat{M})\|_{L^{\frac{10}{3}}(\mathbb F_q^4, dx)}  \|(F_j\ast \widehat{M})\|_{L^2(\mathbb F_q^4, dx)}\\
&+q^{\frac{44}{5}} \sum_{\substack{k=0\\
:q^2\le |F_k|\lesssim q^3}}^\infty \sum_{j=k}^\infty 2^{-k-j} \|(F_k\ast \widehat{M})\|_{L^{\frac{10}{3}}(\mathbb F_q^4, dx)}  \|(F_j\ast \widehat{M})\|_{L^2(\mathbb F_q^4, dx)}.
\end{align*}
Using the upper bound of $\|F_k\ast \widehat{M}\|_{L^{\frac{10}{3}}(\mathbb F_q^4, dx)}$ in \eqref{eq2.15} of Lemma \ref{Lem2.10},  we see that
\begin{align*}
\mbox{B}\le \,& q^{\frac{9}{2}} \sum_{\substack{k=0\\
:1\le |F_k|\le q}}^\infty \sum_{j=k}^\infty 2^{-k-j} |F_k|^{\frac{7}{10}}\,\|(F_j\ast \widehat{M})\|_{L^2(\mathbb F_q^4, dx)}\\
&+ q^{\frac{21}{5}} \sum_{\substack{k=0\\
: q\le |F_k|\le q^2}}^\infty \sum_{j=k}^\infty 2^{-k-j} |F_k| \, \|(F_j\ast \widehat{M})\|_{L^2(\mathbb F_q^4, dx)}\\
&+q^{\frac{24}{5}} \sum_{\substack{k=0\\
:q^2\le |F_k|\lesssim q^3}}^\infty \sum_{j=k}^\infty 2^{-k-j} |F_k|^{\frac{7}{10}} \, \|(F_j\ast \widehat{M})\|_{L^2(\mathbb F_q^4, dx)}\\
&:=B_1+B_2+B_3.\end{align*}
To prove \eqref{eq4.11}, it will be enough to show that
$$ B_1\lesssim 1, \quad B_2\lesssim 1, \quad B_3\lesssim 1.$$
Now recall from \eqref{eq2.16} and \eqref{L2fine} that the estimates
\begin{equation}\label{eq4.14}
\|F\ast \widehat{M}\|_{L^2(\mathbb F_q^4, dx)}
\lesssim q^{-\frac{9}{2}} |F|^{\frac{3}{4}}
\end{equation}
and
\begin{equation}\label{eq4.15} \|F\ast \widehat{M}\|_{L^2(\mathbb F_q^4, dx)} \lesssim
\left\{ \begin{array}{ll} q^{-4} |F|^{\frac{1}{2}}\quad\mbox{if}~~ q^{2} \le |F| \lesssim q^{3}\\
   q^{-5} |F| \quad\mbox{if}~~ q \le |F| \le q^{2}\\
   q^{-\frac{9}{2}} |F|^{\frac{1}{2}} \quad\mbox{if}~~  1\le |F| \le q \end{array}\right. \end{equation}
hold for all $F\subset (C, \sigma) \subset \mathbb F_q^4.$
In order to estimate $B_1$, we use the estimates \eqref{eq4.14}, \eqref{eq4.13}. Then it follows that
\begin{align*} B_1 &\lesssim \sum_{\substack{k=0\\
:1\le |F_k|\le q}}^\infty \sum_{j=k}^\infty 2^{-k-j} |F_k|^{\frac{7}{10}} |F_j|^{\frac{3}{4}} \le \sum_{k=0}^\infty \sum_{j=0}^\infty 2^{-k-j} |F_k|^{\frac{7}{10}} |F_j|^{\frac{3}{4}}\\
&\le \left(\sum_{k=0}^\infty 2^{-\frac{k}{8}}\right) \left(\sum_{j=0}^\infty 2^{-\frac{j}{16}}\right) \sim 1.\end{align*}
Hence, $B_1 \lesssim 1.$ Next, to estimate $B_2,$ we write
\begin{align*}
B_2=& q^{\frac{21}{5}} \sum_{\substack{k=0\\
: q\le |F_k|\le q^2}}^\infty \sum_{\substack{j=k\\
:1\le |F_j|<q^2}}^\infty 2^{-k-j} |F_k| \, \|(F_j\ast \widehat{M})\|_{L^2(\mathbb F_q^4, dx)} \\
&+ q^{\frac{21}{5}} \sum_{\substack{k=0\\
: q\le |F_k|\le q^2}}^\infty \sum_{\substack{j=k\\
: q^2\le |F_j|\lesssim q^3}}^\infty 2^{-k-j} |F_k| \, \|(F_j\ast \widehat{M})\|_{L^2(\mathbb F_q^4, dx)}\\
:=& B_{2,1} +B_{2,2}.
\end{align*}
To estimate $B_{2,1},$ we use \eqref{eq4.14}. Then we see that
$$B_{2,1} \lesssim  q^{-\frac{3}{10}} \sum_{k=0}^\infty \sum_{\substack{j=k\\
:1\le |F_j|<q^2}}^\infty 2^{-k-j} |F_k| \, |F_j|^{\frac{3}{4}}.$$
Observe that if $|F_j|<q^2,$ then $q^{-\frac{3}{10}}|F_j|^{\frac{3}{4}} <|F_j|^{\frac{3}{5}}.$ From this observation, \eqref{eq4.13}, and \eqref{eq4.12}, it follows that
\begin{align*} B_{2,1} &\lesssim  \sum_{k=0}^\infty \sum_{j=k}^\infty
2^{-k-j} |F_k| \, |F_j|^{\frac{3}{5}}\le \sum_{k=0}^\infty 2^{-k}|F_k| \left( \sum_{j=k}^\infty 2^{-\frac{j}{4}}\right)\\
&\sim \sum_{k=0}^\infty 2^{-k}|F_k|  2^{-\frac{k}{4}}=\sum_{k=0}^\infty
2^{-\frac{5k}{4}}|F_k| =1. \end{align*}
In order to estimate $B_{2,2},$ notice from \eqref{eq4.15} that
$$\|(F_j\ast \widehat{M})\|_{L^2(\mathbb F_q^4, dx)} \lesssim q^{-4} |F_j|^{\frac{1}{2}}\quad\mbox{if}~~ q^{2} \le |F_j| \lesssim q^{3}.$$
Using this, we see that
$$B_{2,2} \lesssim q^{\frac{1}{5}}\sum_{k=0}^\infty \sum_{\substack{j=k\\
: q^2\le |F_j|\lesssim q^3}}^\infty 2^{-k-j} |F_k| \, |F_j|^{\frac{1}{2}}.$$
Since $q^{\frac{1}{5}}|F_j|^{\frac{1}{2}} \le |F_j|^{\frac{3}{5}}$ for $|F_j|\ge q^2,$ it follows that
$$B_{2,2} \lesssim \sum_{k=0}^\infty \sum_{j=k}^\infty 2^{-k-j} |F_k| \,
|F_j|^{\frac{3}{5}}.$$
Applying \eqref{eq4.13} and \eqref{eq4.12}, we have
$$B_{2,2} \lesssim \sum_{k=0}^\infty  2^{-k} |F_k| \left(\sum_{j=k}^\infty
2^{-\frac{j}{4}}\right) \sim \sum_{k=0}^\infty  2^{-k} |F_k|\,2^{-\frac{k}{4}}= \sum_{k=0}^\infty
2^{-\frac{5k}{4}}|F_k| =1.$$
Thus, we have proved that $B_2\lesssim 1.$
Finally, we shall prove $B_3\lesssim 1.$
To estimate $B_3$, we split $B_3$ into two terms:
\begin{align*} B_3=& q^{\frac{24}{5}} \sum_{\substack{k=0\\
: q^2\le |F_k|\lesssim q^3}}^\infty \sum_{\substack{j=k\\
:1\le |F_j|<q^2}}^\infty 2^{-k-j} |F_k|^{\frac{7}{10}} \, \|(F_j\ast \widehat{M})\|_{L^2(\mathbb F_q^4, dx)} \\
&+q^{\frac{24}{5}} \sum_{\substack{k=0\\
: q^2\le |F_k|\lesssim q^3}}^\infty \sum_{\substack{j=k\\
:q^2\le |F_j|\lesssim q^3}}^\infty 2^{-k-j} |F_k|^{\frac{7}{10}} \, \|(F_j\ast \widehat{M})\|_{L^2(\mathbb F_q^4, dx)} \\
:=& B_{3,1} +B_{3,2}.
\end{align*}
It follows from \eqref{eq4.14} that
$$ B_{3,1} \lesssim q^{\frac{3}{10}} \sum_{\substack{k=0\\
: q^2\le |F_k|\lesssim q^3}}^\infty \sum_{\substack{j=k\\
:1\le |F_j|<q^2}}^\infty 2^{-k-j} |F_k|^{\frac{7}{10}} \, |F_j|^{\frac{3}{4}}. $$
Since $q^{\frac{3}{10}}|F_k|^{\frac{7}{10}} \le q^{-\frac{3}{10}}|F_k|$ for $q^2\le |F_k|$, we have
$$B_{3,1} \lesssim  \sum_{\substack{k=0\\
: q^2\le |F_k|\lesssim q^3}}^\infty \sum_{\substack{j=k\\
:1\le |F_j|<q^2}}^\infty 2^{-k-j} |F_k| \, q^{-\frac{3}{10}} |F_j|^{\frac{3}{4}}. $$
Using the fact that $q^{-\frac{3}{10}} |F_j|^{\frac{3}{4}} < |F_j|^{\frac{3}{5}}$ for $|F_j|<q^2,$ we obtain that
$$B_{3,1} \lesssim  \sum_{\substack{k=0\\
: q^2\le |F_k|\lesssim q^3}}^\infty \sum_{\substack{j=k\\
:1\le |F_j|<q^2}}^\infty 2^{-k-j} |F_k| \, |F_j|^{\frac{3}{5}}
\le \sum_{k=0}^\infty \sum_{j=k}^\infty 2^{-k-j} |F_k| \, |F_j|^{\frac{3}{5}}.$$
By \eqref{eq4.13} and \eqref{eq4.12}, we see that
\begin{align*} B_{3,1} &\lesssim  \sum_{k=0}^\infty 2^{-k}|F_k| \left( \sum_{j=k}^\infty 2^{-\frac{j}{4}}\right)\\
&\sim \sum_{k=0}^\infty 2^{-k}|F_k|  2^{-\frac{k}{4}}=\sum_{k=0}^\infty
2^{-\frac{5k}{4}}|F_k| =1. \end{align*}
In order to estimate $B_{3,2},$ we begin by recalling from \eqref{eq4.15} that
$$\|(F_j\ast \widehat{M})\|_{L^2(\mathbb F_q^4, dx)} \lesssim q^{-4} |F_j|^{\frac{1}{2}}\quad\mbox{if}~~ q^{2} \le |F_j| \lesssim q^{3}.$$
From this estimate and the definition of $B_{3,2},$ it follows that
$$ B_{3,2} \lesssim q^{\frac{4}{5}} \sum_{\substack{k=0\\
: q^2\le |F_k|\lesssim q^3}}^\infty \sum_{\substack{j=k\\
:q^2\le |F_j|\lesssim q^3}}^\infty 2^{-k-j} |F_k|^{\frac{7}{10}} \, |F_j|^{\frac{1}{2}}.$$
Since $q^{\frac{4}{5}} |F_k|^{\frac{7}{10}} \le q^{\frac{1}{5}} |F_k|$ for $|F_k|\ge q^2,$ it follows that
$$B_{3,2} \lesssim  \sum_{\substack{k=0\\
: q^2\le |F_k|\lesssim q^3}}^\infty \sum_{\substack{j=k\\
:q^2\le |F_j|\lesssim q^3}}^\infty 2^{-k-j} |F_k|\, q^{\frac{1}{5}}|F_j|^{\frac{1}{2}}.$$
We apply a fact that $q^{\frac{1}{5}}|F_j|^{\frac{1}{2}} \le |F_j|^{\frac{3}{5}}$ for $|F_j|\ge q^2,$ and conclude by \eqref{eq4.13} and \eqref{eq4.12} that

\begin{align*}
B_{3,2} &\lesssim \sum_{\substack{k=0\\
: q^2\le |F_k|\lesssim q^3}}^\infty \sum_{\substack{j=k\\
:q^2\le |F_j|\lesssim q^3}}^\infty 2^{-k-j} |F_k|\, |F_j|^{\frac{3}{5}}
\le \sum_{k=0}^\infty \sum_{j=k}^\infty 2^{-k-j} |F_k|\, |F_j|^{\frac{3}{5}}
\\
&\le \sum_{k=0}^\infty 2^{-k}|F_k| \left( \sum_{j=k}^\infty 2^{-\frac{j}{4}}\right)
\sim \sum_{k=0}^\infty 2^{-k}|F_k|  2^{-\frac{k}{4}}=\sum_{k=0}^\infty
2^{-\frac{5k}{4}}|F_k| =1.
\end{align*}
We have proved that $B_3\lesssim 1.$
Putting all estimates together, we complete the proof of the estimate \eqref{Finalproof} for $d=4.$
\end{proof}

\end{document}